\documentclass[english]{amsart}
\usepackage{amsmath}

\usepackage{amssymb}
\usepackage{svn}

\usepackage{amscd}
\usepackage[all,cmtip,line]{xy}

\newcommand{\arr}{{\; \rightarrow \;}}

\SVN $LastChangedDate: 2012-10-30 15:27:28 +0000 (Tue, 30 Oct 2012) $
\SVN $LastChangedRevision: 28 $

\def\BB{\mathbb{B}}

\def\FF{\mathbb{F}}
\def\GG{\mathbb{G}}
\def\HH{\mathbb{H}}

\def\PP{\mathbb{P}}
\def\QQ{\mathbb{Q}}

\def\ZZ{\mathbb{Z}}

\newcommand{\ol}{{\mathcal{O}_L}}

\newcommand{\of}{{\mathcal{O}_F}}

\newcommand{\gera}{{\frak{a}}}
\newcommand{\gerb}{{\frak{b}}}

\newcommand{\gerd}{{\frak{d}}}

\newcommand{\gerl}{{\frak{l}}}
\newcommand{\germ}{{\frak{m}}}

\newcommand{\gerp}{{\frak{p}}}
\newcommand{\gerq}{{\frak{q}}}

\newcommand{\uA}{{\underline{A}}}

\newcommand{\calA}{{\mathcal{A}}}

\newcommand{\calF}{{\mathcal{F}}}

\newcommand{\calH}{{\mathcal{H}}}

\newcommand{\calO}{{\mathcal{O}}}
\newcommand{\calP}{{\mathcal{P}}}

\newcommand{\ra}{\rightarrow}

\newlength{\ownl}

\newcommand{\Cl}{{\operatorname{Cl}}}

\newcommand{\End}{{\operatorname{End}\,}}

\newcommand{\Frob}{{\operatorname{Frob}}}

\newcommand{\Gal}{{\operatorname{Gal}\,}}

\newcommand{\Hom}{{\operatorname{Hom}\,}}

\newcommand{\pr}{{\operatorname{pr}\,}}

\newcommand{\Spec}{{\operatorname{Spec}\,}}

\newcommand{\tr}{{\operatorname{tr}\,}}

\newcommand{\GL}{\operatorname{GL}}

\newcommand{\PSL}{\operatorname{PSL}}

\newcommand{\F}{{\mathbb{F}}}

\newcommand{\Q}{{\mathbb{Q}}}

\newcommand{\cO}{\mathcal{O}}

\newcommand{\tX}{\widetilde{{X}}}
\newcommand{\tY}{\widetilde{{Y}}}

 \newcommand{\rhobar   }{{\overline{\rho}}}

\newcommand{\tgamma   }{\widetilde{\gamma}}

\DeclareMathOperator{\lcm}{lcm}

\newcommand{\Qpbar}{\overline{\Q}_p}

\newcommand{\Fpbar}{\overline{\F}_p}

\newcommand{\Xbar}{\overline{X}}
\newcommand{\Ybar}{\overline{Y}}
\newcommand{\XR}{X^{\rm R}}

\newcommand{\XbarR}{\overline{X}^{\rm R}}
\newcommand{\YbarR}{\overline{Y}^{\rm R}}
\newcommand{\tXbar}{\widetilde{\overline{X}}}
\newcommand{\tYbar}{\widetilde{\overline{Y}}}
\newcommand{\tXbarR}{\widetilde{\overline{X}}^{\rm R}}
\newcommand{\tYbarR}{\widetilde{\overline{Y}}^{\rm R}}
\newcommand{\fbar}{\bar{f}}
\newcommand{\gbar}{\bar{g}}
\newcommand{\Vbar}{\bar{V}}
\usepackage{hyperref} 

\usepackage{amsthm}

\newtheorem{ithm}{Theorem}
\newtheorem{thm}{Theorem}[subsection]

\newtheorem{cor}[thm]{Corollary}
 \newtheorem{lemma}[thm]{Lemma}
\newtheorem{lem}[thm]{Lemma} 
\newtheorem{prop}[thm]{Proposition}
 \theoremstyle{definition}
 \theoremstyle{definition}
\newtheorem{defn}[thm]{Definition} \theoremstyle{remark}

\numberwithin{equation}{subsection}

\theoremstyle{definition}

\begin{document}
\title{Mod $p$ Hilbert modular forms of parallel weight one: the ramified case}

\author{Payman L Kassaei}\email{payman.kassaei@kcl.ac.uk}\address{Department of
  Mathematics, King's College London}  
\begin{abstract} We generalize the main result of ~\cite{GeKa13} to all totally real fields $F$. In other words, for $p>2$ prime, we prove (under a mild Taylor--Wiles hypothesis) that if a modular representation $\rhobar:G_F\to\GL_2(\Fpbar)$ is unramified and $p$-distinguished at all places above~$p$, then it arises from a
  mod~$p$ Hilbert modular form of parallel weight one. This (mostly) resolves the weight one part of Serre's conjecture for totally real fields.
  \end{abstract}
\maketitle\section{Introduction} 

Let $p$ be a prime. Serre conjectured in \cite{Ser87} that if the representation $$\rho:G_\QQ \arr \GL_2(\Fpbar)$$ is continuous, irreducible, and odd, then it is modular in the sense that it arises as the reduction mod $p$ of the $p$-adic Galois representation associated to a Hecke modular eigenform. Serre also formulated a refined conjecture where he specified the minimal weight $k\geq 2$ and level $N$ (coprime to $p$). The weight part of Serre's conjecture was later extended to include the case of weight one mod $p$ modular forms (see \cite{Ed92}); this case is different from the rest as mod $p$ modular forms of weight one are not necessarily reductions of modular forms in characteristic zero. Serre conjectured that a mod $p$ Galois representation arises from a mod $p$ modular form of weight one (and level prime to $p$) if the representation is unramified at $p$. Gross in \cite{Gro90} proved a companion form theorem and settled the weight one part of Serre's conjecture under the hypothesis that the eigenvalues of a Frobenius at $p$ are distinct (the $p$-distinguished assumption). Gorss's proof relied on some unchecked compatibilities, however, a later proof given in \cite{CoVo92} didn't.

The analogues of Serre's refined conjecture in the case of totally real fields have been presented in \cite{BDJ10} and its generalizations to the ramified case. Much work has been done on the weight part of Serre's conjecture, but as in the classical case, the parallel weight one case is different due to lack of automatic liftings to characteristic zero. In ~\cite{GeKa13}, we generalized the result of Gross \cite{Gro90}, and proved a companion forms theorem for Hilbert modular forms of parallel weight one (under the $p$-distinguished assumption), in the case
that~$p$ is {\it unramified} in the totally real field. In this paper we prove the main result  of \cite{GeKa13} for all totally real fields, allowing arbitrary ramification. This (mostly) resolves the weight one part of Serre's conjecture for totally real fields. We should mention that the reverse implications, i.e., the unramifiedness of Galois representations associated to parallel weight one Hilbert modular forms have been studied in \cite{ERX}, \cite{DiWi}.

\begin{ithm}\label{thm: main result, introduction version}
  Let $p>2$ be prime. let $F$ be a totally real field, and $$\rhobar:G_F\to\GL_2(\Fpbar)$$ an irreducible
  modular representation such that $\rhobar|_{G_{F(\zeta_p)}}$ is
  irreducible.   If $p=3$ (respectively $p=5$), assume further that the projective image of
  $\rhobar(G_{F(\zeta_p)})$ is not conjugate to $\PSL_2(\F_3)$ (respectively $\PSL_2(\F_5)$).

Suppose that
  for each place $\gerp|p$, $\rhobar|_{G_{F_\gerp}}$ is unramified, and that
  the eigenvalues of $\rhobar(\Frob_\gerp)$ are distinct. Then, there is a mod $p$ Hilbert modular form $f$ of
  parallel weight $1$ and level prime to $p$ such that $\rhobar_f\cong\rhobar$.
  \end{ithm}

Our method of proof is a combination of modularity lifting theorem
  techniques and geometric methods.  The premise of the argument is the same as in \cite{GeKa13}, however one needs a more detailed study of the geometry of the special fibre of Iwahoric level Hilbert modular varieties in the case $p$ is ramified in $F$ . Using modularity lifting theorems, we obtain  $2^n$ Hilbert modular forms of parallel
  weight $p$ and level prime to $p$, where $n$ is the number of places of $F$
  over  $p$. The mod $p$ reduction of these eigenforms have the same eigenvalue for any Hecke operator away from $p$.   We take a suitable linear combination $\fbar$ of these  mod $p$ Hilbert eigenforms, and show that it is divisible by $H$, the Hasse invariant of parallel weight $p-1$. To do so, we fix a prime $\gerp_0$ above $p$, and a partial Hasse invariant $H_{\beta_0}$ at $\gerp_0$,  and we use the geometry of a stratification on the Rapoport locus of the spacial fibre of the Iwahoric level Hilbert modular variety $Y(\gerp_0)$ to explicitly calculate the quotient of $\fbar$ by $H_{\beta_0}^{e_{\gerp_0}}$. Here, the ramification index $e_{\gerp_0}$ is the maximal power of $H_{\beta_0}$ dividing $H$.    It is easy then to show that the quotient $\fbar/H$ is the desired  Hilbert modular form of parallel weight one.  We use the $p$-distinguished assumption in an essential way: it implies that $\fbar$ is nonzero. However, we expect to be able to remove this condition in an upcoming work. 
  
  \

  \noindent  {\bf Acknowledgments.} We would like to thank Toby Gee for many helpful discussions.

\section{Modularity lifting in weight $p$}\label{sec: BLGG}\subsection{} We first set up some notation.  If $K$ is a field, we let $\overline{K}$ denote
an algebraic closure of $K$. We let $G_K:=\Gal(\overline{K}/K)$
denote the absolute Galois group of $K$. Let $p$ be a prime number, and let
$\varepsilon$ denote the $p$-adic cyclotomic character; our choice of
convention for Hodge--Tate weights is that $\varepsilon$ has all
Hodge--Tate weights equal to $1$. 

Let $F$ be a totally real field. For a finite place $\gerq$ of $F$, we denote by $F_\gerq$ the completion of $F$ at $\gerq$, and by $\calO_{\gerq}$ its ring of integers. We let $\varpi_\gerq$ denote a choice of a uniformizer for $\calO_\gerq$. If
$\gerq$  is coprime to $N$,
then we have the Hecke operator $T_\gerq$
corresponding to the double coset \[\GL_2(\cO_\gerq)
    \begin{pmatrix}
      \varpi_\gerq&0\\0&1
    \end{pmatrix}\GL_2(\cO_\gerq)\]
    acting on the space of Hilbert modular forms of weight $k$ and level $\Gamma_1(N)$. Let 
$f$ be a cuspidal Hilbert modular eigenform in this space. There is a Galois representation $$\rho_f:G_F\to\GL_2(\Qpbar)$$
    associated to $f$, such  that if $\gerq\nmid Np$, and $\Frob_\gerq$ is an \emph{arithmetic} Frobenius element
    of $G_{F_\gerq}$, then $\tr\rho_f(\Frob_\gerq)$ is the $T_\gerq$-eigenvalue of
    $f$, and its determinant  a finite
    order character times $\varepsilon^{k-1}$. We say that
    $$\rhobar:G_F\to\GL_2(\Fpbar)$$ is \emph{modular} if it is equivalent to 
    the reduction mod $p$ of the Galois representation
    $\rho_f:G_F\to\GL_2(\Qpbar)$, for some $f$ as above.

The following is Theorem 2.1.1 in \cite{GeKa13}, proven using modularity lifting techniques and results from  \cite{Ki08}, \cite{CHT08},  \cite{BGHT11}, \cite{Gee11}, \cite{BLGG12},  \cite{GG12},  \cite{Tho12}, \cite{BLGG13}, \cite{BGGT14}.

\begin{thm}\label{thm: lifting to weight p by blgg}

   Let $p>2$ be prime. let $F$ be a totally real field, and $$\rhobar:G_F\to\GL_2(\Fpbar)$$ an irreducible
  modular representation such that $\rhobar|_{G_{F(\zeta_p)}}$ is
  irreducible.   If $p=3$ (respectively $p=5$), assume further that the projective image of
  $\rhobar(G_{F(\zeta_p)})$ is not conjugate to $\PSL_2(\F_3)$ (respectively $\PSL_2(\F_5)$). Suppose that
  for each place $\gerp|p$, \[\rhobar|_{G_{F_\gerp}}\cong 
  \begin{pmatrix}
    \lambda_{\alpha_{1,\gerp}} &0\\0&\lambda_{\alpha_{2,\gerp}}
  \end{pmatrix}\]where $\lambda_{x,\gerp}$ is the unramified character sending
  an arithmetic Frobenius element at $\gerp$ to $x$ (in particular, $\rhobar|_{G_{F_\gerp}}$ is unramified) . For each place $\gerp|p$, let
  $\alpha_\gerp$ be a choice of one of $\alpha_{1,\gerp}$,
  $\alpha_{2,\gerp}$. Let $M$ be an integer coprime to $p$ and divisible
  by the Artin conductor of $\rhobar$. Then there is a Hilbert modular eigenform $f$ of
  parallel weight $p$ such that:
  \begin{itemize}
  \item $f$ has level $\Gamma_1(M)$,
  \item $\rhobar_f\cong\rhobar$,
  \item  for each place
  $\gerp|p$, we have $T_\gerp f=\tilde{\alpha}_\gerp f$, for some lift $\tilde{\alpha}_\gerp$ of
  $\alpha_\gerp$.
  \end{itemize}

\end{thm}
 
 We will use the following corollary of Theorem \ref{thm: lifting to weight p by blgg} in the proof of our main theorem.
 
 \begin{cor}\label{cor:2^n forms} Let notation be as in Theorem \ref{thm: lifting to weight p by blgg}. There is a finite extension $K$ of $\QQ_p$ with residue field $\germ_K$, an integer $N$ divisible by the Artin conductor of $\rhobar$, and, for any $I \subset \{\gerp|p\}$, a normalized Hilbert eigenform $f_I$ of weight $p$ and level
    $\Gamma_1(N)$,  such that
    \begin{enumerate}
    \item $\rhobar_{f_I}\cong\rhobar$,
    \item for each prime
    $\gerp|p$, we have $T_\gerp f_I=\tilde{\alpha}_{I,\gerp}f_I,$ where $\tilde{\alpha}_{I,\gerp}$ lifts
    $\alpha_{1,\gerp}$ if $\gerp\in I$, and $\alpha_{2,\gerp}$ if $\gerp\notin I$,
    \item for any prime $\gerl\nmid p$, the eigenvalues of $T_\gerl$ on the $f_I$'s have all the same reduction in $\germ_K$.
    \end{enumerate} 
    
    \end{cor}
\begin{proof} Let $M$ be as in the statement of Theorem \ref{thm: lifting to weight p by blgg}.  The Theorem guarantees the existence of forms $f_I$ of weight $p$ and level $\Gamma_1(M)$ satisfying the first two conditions.  Since $\rhobar_{f_I}\cong\rhobar$,
    we see that condition (3) is held for all $\gerl nmid Np$. Finally,  by a standard argument, using Proposition 2.3
    of \cite{Shi78}, we can furthermore change the $f_I$'s  (at the possible
    cost of increasing the level from $M$ to $N=M^2$) so that for each prime $\gerl|N$, we have $T_\gerl f_I=0$ for all $I$.  
    \end{proof}

\section{Hilbert Modular Varieties And Modular Forms}\label{section: HMV}

\subsection{Hilbert Modular Varieties of tame level} \label{subsection: HMV} Let $p$ be a prime number.
Let $F/\mathbb{Q}$ be a totally real field of degree $d$, $\mathcal{O}_F$ its ring of integers, and $\gerd_F$ its different ideal.  Let $\mathbb{S}$ be the set of all primes ideals of $\of$ that divide $p$.
For any $\gerp \in \mathbb{S}$, we let $e_\gerp$ denote the ramification index.  We also  let $\FF_\gerp = \of/\gerp$, a finite field of degree $f_\gerp$ over $\FF_p$. Let  $\FF$ denote a finite field with $p^f$ elements, where  $f =\lcm\{f_\gerp: \gerp\vert p\}$. 

 For every $\gerp\in\mathbb{S}$, we let $F_\gerp$ be the completion of $F$ at $\gerp$,  $\calO_\gerp$ its ring of integers, $F_\gerp^{ur}/\QQ_p$ its maximal uramified subextension, and $\calO_\gerp^{ur}$ the ring of integers in $F_\gerp^{ur}$. Fix once and for all a  choice of a uniformizer $\varpi_\gerp \in F_\gerp$ for each $\gerp\in\mathbb{S}$. Let 
 \[
 \mathbb{B}=\textstyle\bigsqcup_{\gerp\in\mathbb{S}}
\mathbb{B}_\mathfrak{p},\] where 
$\mathbb{B}_\mathfrak{p}={\rm Emb}(F_\gerp^{ur},\Qpbar)={\rm Hom}(\FF_\gerp,\FF)$. Let $\Sigma_\gerp={\rm Emb}(F_\gerp,\Qpbar)$. There is a restriction map $res:\Sigma_\gerp \ra \BB_\gerp$ which is $e_\gerp$-to-$1$.

Let $\sigma$ denote the Frobenius automorphism of $W(\FF)$, lifting $x \mapsto
x^p$ modulo $p$. It acts on $\mathbb{B}$ via $\beta \mapsto \sigma
\circ \beta$, and transitively on each $\mathbb{B}_\mathfrak{p}$.
For $S \subseteq \mathbb{B}$, we let
\[\ell(S)=\{\sigma^{-1}\circ\beta\colon \beta \in S\}
\]
\[
r(S)=\{\sigma\circ\beta\colon \beta \in S\},
\]
\[S^c=\mathbb{B} - S.\]

The decomposition \[\mathcal{O}_F
\otimes_\mathbb{Z} W(\FF)=\prod_{p|\gerp} \calO_{F_\gerp} \otimes_{\ZZ_p} W(\FF)=\prod_{p|\gerp}\prod_{\beta \in \mathbb{B_\gerp}}  \calO_{F_\gerp} \otimes_{ \calO_{\gerp}^{ur},\beta} W(\FF)
\]
induces a decomposition,
\[
M=\bigoplus_{\beta\in \mathbb{B}} M_\beta,
\] 
on any $\mathcal{O}_F
\otimes_\mathbb{Z} W(\FF)$-module $M$, where, if $\beta \in \BB_\gerp$, then $\calO_{\gerp}^{ur}$ acts on $M_\beta$ via $\beta:\calO_{\gerp}^{ur}=W(\FF_\gerp) \ra W(\FF)$.

Let $N>3$ be an integer prime to $p$. Let $X/W(\FF)$ be the Hilbert modular scheme classifying Hilbert-Blumenthal abelian schemes (HBAS's)
$\underline{A}/S=(A/S,\iota,\lambda,\alpha)$, where 

\begin{itemize}
\item $S$ is a locally noetherian $W(\FF)$-scheme
\item $A$ is an abelian scheme of relative dimension $d$ over $S$, equipped with
real multiplication $\iota\colon \of \rightarrow \End_S(A)$,
\item  $\lambda$ is a polarization as in \cite{DP94},
namely, over each connected component $T$ of $S$, an isomorphism $\lambda\colon (\calP_{A|_T}, \calP_{A|_T}^+)
\rightarrow (\gera, \gera^+)$ for some representative $(\gera,
\gera^+)\in [\Cl^+(L)]$ (depending on $T$) such that $A|_T \otimes_\ol \gera \cong
A|_T^\vee$. Here, $\calP_A = \Hom_\of(A, A^\vee)^{\rm
sym}$, viewed as sheaf of projective $\of$-modules of rank one in the \'etale topology, and $\calP^+_A$ is the cone of polarizations.
\item $\alpha$ is a rigid
$\Gamma_{00}(N)$-level structure,  that is, $\alpha\colon \mu_N
\otimes_{\ZZ} \gerd_F^{-1} \arr A$ is an $\of$-equivariant
closed immersion of group schemes.
\end{itemize}

The Hilbert modular scheme $X$ is a normal scheme of relative dimension $d$ over $\Spec(W(\FF))$. Let $\Xbar=X \otimes_{W(\FF)} \FF$ be the special fibre, viewed as a closed subscheme of $X$. It represents the same functor restricted to the subcategory of locally noetherian $\FF$-schemes.  Let $\tX$ denote a toroidal compactification of $X$, and $\tXbar$ be its special fibre. For a fractional ideal $\gera$ of $\of$, we denote by $X_\gera$ the locus in $X$ where the polarization module is isomorphic to $(\gera,\gera^+)$. Then $X_\gera$ is an irreducible component of $X$, and every irreducible component of $X$ is of the form $X_\gera$ for some fractional ideal $\gera$. We define $\Xbar_\gera$ similarly. We define $\tX_\gera$, $\tXbar_\gera$, respectively, by taking the Zariski closure of $X_\gera$, $\Xbar_\gera$ in their toroidal compactifications.

Let $\epsilon: {\mathcal A}_X \arr X$ be the universal abelian scheme. Let 
 \[
 \Omega=\epsilon_\ast \Omega^1_{{\mathcal A}_X/X}
\]
\[
\HH=\epsilon_\ast H^1_{dR}({{\mathcal A}_X/X}),
\]
which are locally free sheaves of rank, respectively, $d,2d$ on $X$. We denote the restriction of these sheaves to $\Xbar$ by the same notation.  As explained above these $(\of \otimes_\ZZ W(\FF))-$modules decompose as 
\[
\Omega=\bigoplus_{\beta\in \mathbb{B}} {\Omega}_\beta,
\]
\[
\HH=\bigoplus_{\beta\in \mathbb{B}} {\HH}_\beta,
\]
where $\Omega_\beta$, $\HH_\beta$ are locally free sheaves of rank, respectively,  $e_\gerp$, $2e_\gerp$, if $\beta \in \BB_\gerp$. We define
\[
\omega_\beta:={\Omega}_\beta/\varpi_\gerp{\Omega}_\beta,
\]
\[
\calH_\beta=\HH_\beta/\varpi_\gerp {\HH}_\beta,
\]
where $\beta \in \BB_\gerp$. Again, we use the same notation to denote the restriction of these sheaves to $\Xbar$. These sheaves extend naturally to the toroidal compactifications  of $X$ and $\Xbar$, and we denote the extensions by the same notation. For each $\beta \in \BB$, $\calH_\beta$ is a locally free sheaf of rank $2$ over $\tXbar$. Similarly for any $\uA/S$ classified by $X$, we can define $\Omega_{A,\beta}$, $\HH_{A,\beta}$, $\omega_{A,\beta}$, $\calH_{A,\beta}$ to be $\calO_S$-modules.

Let $\XR$ denote the locus of points $\uA/S$ on $X$ which satisfy the Rapoport condition, namely, where $\Omega$ is locally free of rank one as a module over $\calO_F \otimes_{\ZZ} \calO_S$. The Rapoport locus can be shown to be the open subscheme of $X$ on which it is smooth over $\Spec(W(\FF))$. Let $\XbarR=\Xbar \cap \XR$. Then, $\Xbar-\XbarR$ has codimension at least  $2$ (and is empty if $p$ is unramified in $F$).  We denote by $\tXbarR$ the union of $\XbarR$ and the cuspidal locus of $\Xbar$. For a fractional ideal $\gera$ of $\of$, we let $\tXbarR_\gera=\tXbarR \cap \tXbar_\gera$.

The Rapoport condition is equivalent to $\omega_\beta$ being a locally free sheaf of rank one for all $\beta \in \BB$. In fact, over $\tXbarR$, each $\omega_\beta$ is a  locally free subsheaf of $\calH_\beta$ of rank one.

\begin{defn} Let $R$ be an $\FF$-algebra. The space of mod $p$ Hilbert modular forms of parallel weight $k\in \ZZ$, and level $\Gamma_{1}(N)$ over $R$ is defined to be
\[
M_k(N;R)=H^0(\Xbar \otimes_{\FF} R, (\wedge^d\ \Omega)^k).
\]
Since $d>1$, the Koecher principle shows that $M_k(N;R)=H^0(\tilde{\Xbar} \otimes_{\FF} R, (\wedge^d\ \Omega)^k)$. We can recast this definition in terms of the $\omega_\beta$'s.
\end{defn}
 
We can recast the above definition in terms of the $\omega_\beta$'s:

\begin{lem} \label{Lemma: sections} Let $k \in \ZZ$. We have
\[
M_k(N,R) \cong H^0(\XbarR \otimes_{\FF} R,\displaystyle{\bigotimes_{\gerp|p}\bigotimes_{\beta\in \BB_\gerp} }\omega_\beta^{e_\gerp k})=H^0(\tXbarR \otimes_{\FF} R,\displaystyle{\bigotimes_{\gerp|p}\bigotimes_{\beta\in \BB_\gerp} }\omega_\beta^{e_\gerp k}).
\]

\end{lem}

\begin{proof}
 Since $\Xbar$ is normal, and $\Xbar-\XbarR$ has codimension at least  $2$, we have $M_k(N,R)=H^0(\XbarR \otimes_{\FF} R, (\wedge^d\ \Omega)^k)$. We have $$\wedge^d \Omega \cong \bigotimes_{\gerp|p}\bigotimes_{\beta\in \BB_\gerp} \wedge^{e_\gerp} \Omega_\beta.$$ It is enough to show that $\wedge^{e_\gerp} \Omega_\beta \cong \omega_\beta^{e_\gerp}$ on $\XbarR$. Over $\XbarR$, we have a filtration on $\Omega_\beta$
 \[
 0=\varpi_\gerp^{e_\gerp}\Omega_\beta \subset \varpi_\gerp^{e_\gerp-1}\Omega_\beta \subset \cdots \subset \varpi_\gerp\Omega_\beta \subset \Omega_\beta.
\]
 Since on  $\XbarR$ the sheaf $\omega_\beta=\Omega_\beta/\varpi_\gerp \Omega_\beta$ is locally free of rank one, all the graded parts of the above filtration are isomorphic to $\omega_\beta$. This shows that $\wedge^{e_\gerp} \Omega_\beta \cong \omega_\beta^{e_\gerp}$, and completes the proof of the lemma.
\end{proof}

An example  of a mod $p$ Hilbert modular form is given by the Hasse invariant defined as follows. Let $\overline{\calA}/\Xbar$ be the universal HBAS over $\Xbar$. Let 
\[
V:\calA^{(p)} \rightarrow \calA
\]
be the Verschiebung morphism. Then $V^*: \wedge^d\ \Omega \rightarrow  (\wedge^d\ \Omega)^{(p)}$ defines a section of $(\wedge^d\ \Omega)^{p-1}$ on $\Xbar$ which we call the Hasse invariant and denote by $H$. It can also be defined in terms of the partial Hasse invariant defined as follows.

The morphism $V^*: \Omega \rightarrow  \Omega^{(p)}$ decomposes as 
\[
V^*=\oplus_{\beta \in \BB} V_\beta^*,
\]
 where the components are morphisms  $V^*_\beta: \Omega_\beta \rightarrow \Omega^{(p)}_{\sigma^{-1}\circ\beta}$. Let $$\Vbar^*_\beta: \omega_\beta \rightarrow \omega^{(p)}_{\sigma^{-1}\circ\beta}$$ denote the reduction of $V^*_\beta$ module $\pi_\gerp$, where $\beta \in \BB_\gerp$. Then $\Vbar^\ast_\beta$ defines a section of $\omega^{p}_{\sigma^{-1}\circ\beta}\otimes \omega^{-1}_\beta$ on $\Xbar$ which is denoted $H_\beta$.

\begin{lem} We have $H=\displaystyle{\prod_{\gerp|p}\prod_{\beta \in \BB_\gerp} H_\beta^{e_\gerp}}$.

\end{lem}

\begin{proof} This can be seen just as in the proof of Lemma \ref{Lemma: sections}. All we need to note is that over $\XbarR$,  the  morphism induced by  $V_\beta^*$ on each graded part of the filtration given in the proof of Lemma \ref{Lemma: sections} is the same as $\Vbar_\beta^*$ after identifying the graded part with $\omega_\beta$.

\end{proof}

We consider the  Goren-Oort stratification on $\tXbarR$ defined as follows. For any $T \subset \BB$, we define 
\[
Z_T=V(h_\beta: \beta \in T),
\]
where $V(h_\beta: \beta \in T)$ denotes the closed subscheme given by the vanishing of the ideal generated by $H_\beta$'s for $\beta \in T$. For $\beta \in \BB$, we denote $Z_{\{\beta\}}$ by $Z_\beta$.
Each $Z_T$ is non-empty, nonsingular, and equi-dimensional of dimension $d - |T|$. The divisor of $H$ on $\tXbarR$ is $\displaystyle\sum_{\gerp \in \mathbb{S}}\displaystyle\sum_{\beta \in \BB_\gerp}e_\gerp Z_{\beta}$, where the $Z_\beta$'s are normal crossing divisors \cite[Corollary 8.18]{AG05}.

The ordinary locus in $\tXbarR$ equals $\tXbar^{\rm ord}=\tXbarR-\bigcup_{\beta \in \BB} Z_\beta$. For a fractional ideal $\gera$ of $\of$, we define $\tXbar^{\rm ord}_\gera=\tXbar^{\rm ord} \cap \tXbar_\gera$.

\subsection{Hilbert Modular Varieties with Iwahori level} \label{subsection: IwahoriHMV}

Let notation be as in the previous section, and fix $\gerp_0$ a prime of $\of$ above $p$. Let $Y=Y(\gerp_0)/W(\FF)$ be the Hilbert modular scheme classifying the isomorphism classes of $(\delta\colon\uA
\arr \uA')$, where 

\begin{itemize}
\item $\uA/S = (A/S, \iota_A, \lambda_A, \alpha_A)$, $\uA'/S =
(A'/S, \iota_{A'}, \lambda_{A'}, \alpha_{A'})$ are classified by $X$;

\item $\delta$ is an $\of$-isogeny of degree $p^{f_{\gerp_0}}$ with kernel contained in $A[\gerp_0]$, such that $\delta^\ast \calP_{A'}= \gerp_0\calP_A$.
\end{itemize}

The Hilbert modular scheme $Y$ is a normal scheme of relative dimension $d$ over $\Spec(W(\FF))$. Let $\Ybar=Y \otimes_{W(\FF)} \FF$ be the special fibre.   Let $\tY$ denote a toroidal compactification of $Y$, and $\tYbar$ be its special fibre. We denote by the universal object on $Y$ by $\delta: \calA_Y \rightarrow \calA'_Y$. Since kernel of $\delta$ is contained in $\calA_Y[\gerp_0]$, we can construct an isogeny 
\[
\delta':\calA'_Y\cong \calA_Y/{\rm Ker}(\delta) \ra \calA_Y/\calA_Y[\gerp_0]\cong \calA_Y \otimes_\of \gerp_0^{-1}.
\]
It is easy to see that $\delta' \circ \delta: \calA_Y \ra \calA_Y \otimes_\of \gerp_0^{-1}$ is the natural map induced by the inclusion of $\of$ in $\gerp_0^{-1}$. Similarly, $\delta \circ (\delta' \otimes_\of \gerp_0): \calA'_Y \otimes_\of \gerp_0 \ra \calA'_Y$ is the natural map induced by the inclusion of $\gerp_0$ in $\calO_F$. In particular, for any point $(\delta: \uA \ra \uA')$ on $Y$, there is $\delta': \uA' \ra \uA\otimes_{\of} \gerp_0^{-1}$ satisfying a similar property.

There are two  maps
\[
\pi_{1},\pi_2:\tYbar \ra \tXbar
\]
where on the noncuspidal locus $\pi_1(\delta\colon \uA\ra \uA')=\uA$, and $\pi_2(\delta\colon \uA\ra \uA')=\uA'$. 

For any $\beta \in \BB$, let 
\[
\rm{pr}_\beta^\ast: \pi_2^\ast {\omega_\beta} \rightarrow \pi_1^\ast{\omega_\beta}
\]
 denote the map  induced by pulling back differential one forms under the natural projection  $\delta: \calA_Y \rightarrow \calA_Y'$.

We now define a stratification on $\tYbarR$. We first define it on the noncuspidal locus  $\YbarR$. Let $Q=(\delta\colon \uA\ra \uA')$ be a closed point on $\YbarR$. Recall the map $\delta':\uA' \ra \uA \otimes_\of \gerp_0^{-1}$ constructed above. The pullback morphisms $\delta^*: \Omega_{A'} \ra \Omega_A$ and $\delta'^*:\Omega_{A\otimes_{\of} \gerp_0^{-1}}\ra \Omega_{A'}$ induce morphisms
\begin{align}
\delta_\beta^*  \colon \omega_{A',\beta}& \ra \omega_{A,\beta},
\\\nonumber
{\delta'}_\beta^*  \colon \omega_{A\otimes_\of \gerp_0^{-1},\beta}&\ra \omega_{A',\beta}.
\end{align}

We define
\begin{align}
\varphi(Q)&=\{\beta \in \BB_{\gerp_0}: {\delta}_{\sigma^{-1}\circ\beta}^*=0 \},
\\\nonumber
\eta(Q)&=\{\beta \in \BB_{\gerp_0}: {\delta'}_\beta^*=0 \}.
\end{align}
Note that since $(\delta'\circ\delta)^*=0$, we have $\ell(\varphi(Q)) \cup \eta(Q)=\BB_{\gerp_0}$.

\begin{prop}\label{Proposition: stratification}
Let $\varphi, \eta$ be subsets of $\BB_{\gerp_0}$ such that $\ell(\varphi) \cup \eta=\BB_{\gerp_0}$.

\begin{enumerate}
 \item There is a locally closed subset $W_{\varphi,\eta}$
of $\tYbarR$ with the following property: a closed point $Q$ of
$\tYbarR$ lies in $W_{\varphi,\eta}$ if and only if $\varphi(Q)=\varphi$, and $\eta(Q)=\eta$.
Moreover, the subset
\[ Z_{\varphi,\eta} = \bigcup_{\varphi'\supseteq\varphi , \eta'\supseteq\eta} W_{\varphi', \eta'}\]
is closed.
\item $W_{\varphi,\eta}$ is non-empty, and its Zariski closure is $Z_{\varphi,\eta}$. The
collection $\{W_{\varphi,\eta}\}$ is a stratification of $\tYbarR$ by $3^{f_{\gerp_0}}$ strata.
\item $W_{\varphi,\eta}$ and $Z_{\varphi,\eta}$ are nonsingular, equi-dimensional, and
\[ \dim (W_{\varphi,\eta}) = \dim(Z_{\varphi,\eta}) = d+f_{\gerp_0} - (|\varphi| + |\eta|).\]
\item The irreducible components of $\tYbarR$ are the irreducible
components of the strata $Z_{\varphi, \ell(\varphi^c)}$ for $\varphi
\subseteq \BB$.
\item $\pi_1(Z_{\varphi,\eta})=Z_{\varphi \cap \eta}$, and the map $\pi_1: Z_{\varphi,\eta} \ra Z_{\varphi \cap \eta}$ is surjective. In fact, every fibre of this map is homeomorphic to $(\PP^1)^{|\varphi|\cap|\eta|}$.

\end{enumerate}
\end{prop}

\begin{proof} The stratification defined above is a direct generalization of the one defined and studied in \cite{GoKa12}  on $\tYbar$ in the case $p$ is unramified in $\of$ to the Rapoport locus in the general case. One can prove the above results (and much more) directly generalizing the proofs in \cite{GoKa12}, essentially by replacing $\HH_\beta$ with $\calH_\beta$ in the study of the deformation theory in the general case (See Theorem 2.5.2, and \S 2.6 in {\it loc. cit.}). 

Alternatively, one can appeal to \S 3.4, 3.5 of \cite{ERX} where the stratification in \cite{GoKa12} and some of its properties have been generalized to the Pappas-Rapoport model of the Hilbert modular variety $\tYbar^{\rm PR}$ that maps to $\tYbar$ via a surjective morphism which is an isomorphism on the Rapoport locus. For reference, $Z_{\varphi,\eta}$ is denoted by $Y_{S,S'}$ (intersected with the Rapoport locus), with $S=res^{-1}(\ell(\varphi))$, $S'=res^{-1}(\eta)$ (intersected with the Rapoport locus), where $res:\Sigma_\gerp \ra \BB_\gerp$ is the restriction map.
\end{proof}

\begin{cor} \label{Corollary: pi_1 inverse}The morphism $\pi_1:Z_{\emptyset,\BB_{\gerp_0}} \ra \tXbarR$ is surjective, and establishes a bijection between the irreducible components of $Z_{\emptyset,\BB_{\gerp_0}}$ and $\tXbarR$.
\end{cor}

\begin{proof} This follows from parts (3) and (5) of Proposition \ref{Proposition: stratification}. 
\end{proof}

\begin{lem} \label{Lemma: Isomorphism} Over $Z_{\emptyset,\BB_{\gerp_0}}\subset \tYbarR$, for every $\beta\in\BB_{\gerp_0}$, there is an isomorphism $\mu_\beta$ that makes the following diagram commutative.
 \begin{equation}\label{Diagram: isom}
 \xymatrix{  \pi_2^*\omega_{\sigma^{-1}\circ\beta}^p \ar[r]^{\mu_\beta}\ar[rd]_{(\pr_{\sigma^{-1}\circ\beta}^*)^p} & \pi_1^*\omega_\beta \ar[d]^{\Vbar^*_\beta} \\
& \pi_1^*\omega_{\sigma^{-1}\circ\beta}^p.}
 \end{equation}
\end{lem}

\begin{proof} The Verschiebung morphism of $\calA'_Y$ induces $\Vbar^*_\beta:\pi_2^*\calH_\beta \ra \pi_2^*\calH_{\sigma^{-1}\circ\beta}^{(p)}$ which extends $\Vbar^*_\beta:\pi_2^*\omega_\beta \ra \pi_2^*\omega_{\sigma^{-1}\circ\beta}^{(p)}$, and which has image equal to $\pi_2^*\omega_{\sigma^{-1}\circ\beta}^{(p)}$. It follows that its kernel is a locally free subsheaf of rank one of $\pi_2^*\calH_\beta$.

We also have $\delta_\beta^*:\pi_2^*\calH_\beta \ra \pi_1^*\calH_\beta$ and ${\delta'}_\beta^*: \pi_1^*\calH_{\calA_X \otimes_\of \gerp_0^{-1},\beta} \ra \pi_2^*\calH_\beta$ induced by the universal  isogenies $\delta$ and $\delta'$. Let $\beta \in \BB_{\gerp_0}$, and $(\delta:\uA \ra \uA')$ be a point on $\YbarR$. Interpreting $\calH_{A,\beta}$, $\calH_{A',\beta}$ as the $\beta$-components of the contravariant Dieudonne modules of $\calA[\gerp_0]$ and $\calA'[\gerp_0]$, and using the Rapoport condition, it follows that both $\delta_\beta^*\colon\calH_{A',\beta} \ra \calH_{A,\beta}$, and ${\delta'}_\beta^*\colon\calH_{A\otimes_\of {\gerp^{-1}_0},\beta} \ra \calH_{A',\beta}$ have $1$-dimensional kernels and images. Since $\delta_\beta^*{\delta'}_\beta^*=0$, we deduce that (for the universal data) we have ${\rm Ker}(\delta_\beta^*)={\rm Im}({\delta'}_\beta^*)$. In particular, since over $Z_{\emptyset,\BB_{\gerp_0}}\subset \tYbarR$ we have ${\delta'}_\beta^*(\pi_1^*\omega_{\calA_X \otimes_\of \gerp_0^{-1},\beta})=0$ for all $\beta \in \BB_{\gerp_0}$, it follows that ${\rm Im}(\delta^*_\beta)=\pi_1^*\omega_\beta$ over $Z_{\emptyset,\BB_{\gerp_0}}$. 

Putting the above together, we obtain a commutative diagram as follows 

 \begin{equation}\label{Diagram: isom}
 \xymatrix{  \pi_2^*\calH_{\beta} \ar@{>>}[rr]^{\delta^*_\beta}\ar@{>>}[d]_{\Vbar^*_\beta} && \pi_1^*\omega_\beta \ar[d]^{\Vbar^*_\beta} \\
\pi_2^*\omega_{\sigma^{-1}\circ\beta}^{(p)} \ar[rr]_{(\delta^*_{\sigma^{-1}\circ\beta})^{(p)}}&& \pi_1^*\omega_{\sigma^{-1}\circ\beta}^{(p)}.}
 \end{equation}
To end the proof, it is enough to show that the two surjective maps in the diagram, $\delta_\beta^*$, $\Vbar^*_\beta$, have the same kernel, in which case, one can define $\mu_\beta:=\delta^*_\beta \circ ({{\Vbar}_\beta}^*)^{-1}$. We have seen that ${\rm Ker}({\delta_\beta}^*)={\rm Im}({\delta'}_\beta^*)$. Since both kernels in question are locally free of rank one, it is enough to show that ${\rm Im}({\delta'}_\beta^*)\subset {\rm Ker}(\Vbar_\beta^*)$, i.e., $\Vbar_\beta^*\circ{\delta'}^*_\beta=0$. This follows from the following commutative diagram

 \begin{equation}\label{Diagram: isom}
 \xymatrix{ \pi_1^*\calH_{\calA_X \otimes_\of \gerp_0^{-1},\beta} \ar[rrr]^{{\delta'}_\beta^*}\ar@{>>}[d]_{ \Vbar_\beta^*} &&& \pi_2^*\calH_{\beta}  \ar@{>>}[d]^{\Vbar_\beta^*}  \\
(\pi_1^*\omega_{\calA_X \otimes_\of \gerp_0^{-1},\sigma^{-1}\circ\beta})^{(p)}\ar[rrr]_{({\delta'}^*_{\sigma^{-1}\circ\beta})^{(p)}}&&& \pi_2^*\omega_{\sigma^{-1}\circ\beta}^{(p)}  }
 \end{equation}
 shows that $\Vbar_\beta^*\circ{\delta'}^*_\beta=({\delta'}^*_{\sigma^{-1}\circ\beta})^{(p)}\circ \Vbar_\beta^*=0$ since ${\delta'}^*_{\sigma^{-1}\circ\beta}=0$ over $Z_{\emptyset,\BB_{\gerp_0}}$.
\end{proof}

\begin{lem}\label{Lemma: Hasse} Let $\beta \in \BB$. For any sheaf $\calF$ on $\tYbarR$,
\[
(\bar{V}_\beta^*)^{n_\beta}\otimes 1:H^0(\pi_1^*\omega_\beta^{n_\beta}\otimes \calF) \ra H^0(\pi_1^*\omega_\beta^{pn_\beta}\otimes \calF)
\]
is multiplication by $\pi_1^*H_\beta^{n_\beta}$.
\end{lem}

\begin{proof} This is immediate from the definition of $H_\beta$.

\end{proof}
 
\subsubsection{Tate objects and $q$-expansions}\label{section: q-exp} Given any fractional ideal $\gera$ of $\calO_F$, let $X_\gera$  denote the subscheme of $X$  where the polarization module of the abelian scheme $(\calP_A,\calP^+_A)$ is isomorphic to $(\gera,\gera^+)$ as a module with notion of positivity. $X_\gera$ is a connected component of  $X$, and every connected component of $X$ is of the form $X_\gera$ for some $\gera$ as above.   Similarly, we can define $Y_\gera$ (where, for a point $(\delta:\uA \ra \uA^\prime)$, we impose the condition on the polarization module of $\uA$) and a similar statement is true for $Y$. 

 In the following, we refer to \cite[\S6]{AG05} for Tate objects over Hilbert modular varieties and $q$-expansions for Hilbert modular forms, even though our notation will be  slightly different.  For example, for a pair of fractional ideals $\gera,\gerb$, one can define a map 
 \[
 \underline{q}:  \gera^{-1} \arr \GG_m\!\otimes_\ZZ\! \gerd_F^{-1}\gerb^{-1}
 \]
 by sending an element $\alpha \in \gera^{-1}$ to the point of the torus $\GG_m\! \otimes_\ZZ\! \gerd_F^{-1}\gerb^{-1}=\GG_m\! \otimes_\ZZ\! \gerb^{*}$ whose value at the parameter $X^{\xi}$ (for $\xi \in \gerb$) is $q^{\xi\alpha}$ (over an appropriate ring which contains all these elements). We will replace the notation $\underline{q}(\gera^{-1})$ employed by \cite{AG05}  with $q^{\gera^{-1}}$. Let 
 \[
Ta_{\gera,\gerb}=(\underline{\GG_m\! \otimes_\ZZ\! \gerd_F^{-1}\gerb^{-1})/q^{\gera^{-1}}}
\]
denote a cusp on $X_{\gera\gerb}$, where underline indicates the inclusion of standard PEL structure (c.f. \cite[\S 6.4]{AG05}). For simplicity, we denote $Ta_{\gera,\of}$ by $Ta_\gera$. 
Let $\Omega(Ta_{\gera,\gerb})$ denote the invariant differentials on $Ta_{\gera,\gerb}$ defined over a base $\Spec(R)$. We have
\[
\Omega(Ta_{\gera,\gerb}) \cong ( (\gerd_F^{-1}\gerb^{-1})^* \otimes_\ZZ R) \frac{dt}{t}\cong (\gerb \otimes_\ZZ R)  \frac{dt}{t},
\]
where $t$ is the parameter on $\GG_m$. 

 Let $\gerp_0$ be a prime above $p$ in $\of$. Then
\[
Ta'_{\gera}=(\delta\colon Ta_\gera \ra Ta_{\gerp_0\gera})
\]
is a cusp on $Y_\gera$, where $\delta$ is the natural projection.

\begin{lem}\label{Lemma; cusps} Let $\gera$ be a fractional ideal of $\of$. Every cusp of the form $Ta'_\gera$ belongs to $W_{\emptyset,\BB_{\gerp_0}}$. Every irreducible (equivalently, connected) component of $Z_{\emptyset,\BB_{\gerp_0}}$ contains a cusp of the form $Ta'_\gera$.
\end{lem}

\begin{proof}

 Note that  the morphism
\[
\delta'\colon Ta_{\gerp_0\gera} \ra Ta_{\gera}\otimes_\of \gerp_0^{-1} \cong Ta_{\gerp_0\gera,\gerp_0},
\]
defined at the beginning of \S\ref{subsection: IwahoriHMV},  is the natural map $Ta_{\gerp_0\gera} \ra Ta_{\gerp_0\gera,\gerp_0}$. The induced morphism on the invariant differentials, over a base $\Spec(R)$, is given by  the natural inclusion
\[
\delta'^*\colon \Omega(Ta_{\gerp_0\gera,\gerp_0})\cong (\gerp_0 \otimes_{\ZZ} R) \frac{dt}{t} \ra  \Omega(Ta_{\gerp_0\gera})\cong (\of \otimes_{\ZZ} R) \frac{dt}{t}.
\]
This shows that the image of $\delta'^*$ is contained in the image of the action of $\gerp_0$. It follows that $\delta'^*_\beta=0$ for all $\beta \in \BB$. Similarly, one can show that $\delta^*_\beta \neq 0$ for all $\beta \in \BB$. This proves our claim that $Ta'_{\gera} \in W_{\emptyset,\BB}$.  For the second statement, let $C$ be an irreducible component of $Z_{\emptyset,\BB_{\gerp_0}}$. By Corollary \ref{Corollary: pi_1 inverse}, $\pi_1(C)=\tXbarR_\gera$ for some fractional ideal $\gera$.   Since $Ta_\gera \in \tXbarR_\gera$, and $Ta'_\gera \in Z_{\emptyset,\BB_{\gerp_0}}$ it follows that $Ta'_\gera \in \pi_1^{-1}(\tXbarR_\gera) \cap Z_{\emptyset,\BB_{\gerp_0}}$, which equals $C$ by Corollary \ref{Corollary: pi_1 inverse}.

\end{proof}

We let $\underline{\omega}$ denote the canonical generator of the sheaf $\wedge^d\ \Omega(Ta_\gera)$ on the base of $Ta_\gera$ .  If $f$ is a normalized Hilbert modular eigenform of
  parallel weight $k$, we can write
  \[
  f(Ta_\gera)=\sum_{\xi \in (\gera^{-1})^+\cup \{0\}}
  c_\xi q^\xi\underline{\omega}^k,
  \]
and in this representation $c_\xi=c(\xi\gera,h)$ is the eigenvalue of the $T_{\xi\gera}$ operator on $f$, for all $\xi\in (\gera^{-1})^+$. See \cite[(2.23)]{Shi78}.

\

\section{The main Theorem} 

\subsection{The statement} We now state our main result.

\begin{thm}\label{thm: main result}   Let $p>2$ be prime. let $F$ be a totally real field, and $$\rhobar:G_F\to\GL_2(\Fpbar)$$ an irreducible
  modular representation such that $\rhobar|_{G_{F(\zeta_p)}}$ is
  irreducible.  
If $p=3$ (respectively $p=5$), assume further that the projective image of
  $\rhobar(G_{F(\zeta_p)})$ is not conjugate to $\PSL_2(\F_3)$ (respectively $\PSL_2(\F_5)$).
  
Suppose that for each prime $\gerp|p$, $\rhobar|_{G_{F_\gerp}}$ is unramified, and that
  the eigenvalues of $\rhobar(\Frob_\gerp)$ are distinct.

 Then, there is a mod $p$ Hilbert modular form $h$ of
  parallel weight $1$ and level prime to $p$ such that
  $\rhobar_h\cong\rhobar$. Furthermore, $h$ can be chosen to have
  level bounded in terms of the Artin conductor of $\rhobar$.
  \end{thm}

\subsection{The proof} We present the proof of our main theorem in several steps. By assumption, for each $\gerp\in \mathbb{S}$, the Frobenius at $\gerp$ has two distinct eigenvalues which we denote by  $\alpha_{1,\gerp},\alpha_{2,\gerp}$.  By Corollary \ref{cor:2^n forms}, there is a finite extension $K$ of $\QQ_p$ with ring of integers $\calO_K$ and residue field $\germ_K$, an integer $N$ divisible by the Artin conductor of $\rhobar$, and, for any $I \subset \{\gerp|p\}$, a normalized Hilbert eigenform $f_I$ of weight $p$ and level
    $\Gamma_1(N)$,  such that
    \begin{itemize}
    \item $\rhobar_{f_I}\cong\rhobar$,
    \item for each prime
    $\gerp|p$, we have $T_\gerp f_I=\tilde{\alpha}_{I,\gerp}f_I,$ where $\tilde{\alpha}_{I,\gerp}$ lifts
    $\alpha_{1,\gerp}$ if $\gerp\in I$, and $\alpha_{2,\gerp}$ if $\gerp\notin I$,
    \item for any prime $\gerl \nmid p$, the eigenvalues of $T_\gerl$ on the $f_I$'s have all the same reduction in $\germ_K$.
    \end{itemize}

For any $I\subset \mathbb{S}$, let $\tilde{\alpha}_I=\Pi_{\gerp|p}
    \tilde{\alpha}_{I,\gerp} $. Set
    \[
    f=\sum_{I\subset \mathbb{S}}(-1)^{|I|}\tilde{\alpha}_I f_I,
    \]
    \[
    g=\sum_{I\subset \mathbb{S}}(-1)^{|I|}f_I.
    \]Taking $N$ as above, and extending scalars from $W(\FF)$ to $W(\FF')$ large enough to contain $\calO_K$, we can view $f,g$ as elements of  $H^0(\tX \otimes W(\FF'),(\wedge^d\Omega)^p)$. For simplicity of notation, we won't indicate the base change in our notation, and write the rest of the proof as if the forms are defined over $W(\FF)$. Let $\fbar,\gbar$ denote their images in $H^0(\tXbar,(\wedge^d\Omega)^p)$. By Lemma \ref{Lemma: sections}, we can view
\[
\fbar,\gbar \in H^0(\tXbarR ,\displaystyle{\bigotimes_{\gerp|p}\bigotimes_{\beta\in \BB_\gerp} }\omega_\beta^{e_\gerp p}).
\]

\begin{defn} We define the following morphism of sheaves on $\tYbarR$,
\[
\pr^*:=\bigotimes_{\gerp|p}\bigotimes_{\beta\in \BB_\gerp} (\pr^*_\beta)^{e_\gerp}: H^0(\tYbarR ,\displaystyle{\bigotimes_{\gerp|p}\bigotimes_{\beta\in \BB_\gerp} }\pi_2^*\omega_\beta^{e_\gerp p}) \ra H^0(\tYbarR ,\displaystyle{\bigotimes_{\gerp|p}\bigotimes_{\beta\in \BB_\gerp} }\pi_1^*\omega_\beta^{e_\gerp p}),
\] 
where $\pr_\beta^*$ is defined in \S \ref{subsection: HMV}. Under the isomorphism in Lemma \ref{Lemma: sections}, this morphism is simply 
\[
(\wedge^d\delta^*)^p: (\wedge^d \Omega_{\calA'_{\tYbarR/\tYbarR}})^p \ra  (\wedge^d \Omega_{\calA_{\tYbarR/\tYbarR}})^p,
\]
induced by the universal isogeny  $\delta: \calA_Y \rightarrow \calA'_Y$.
\end{defn}

Both $\pi_1^*f$ and $\pr^*\pi_2^*g$ can be viewed as sections in $H^0(\tYbarR,\displaystyle{\bigotimes_{\gerp|p}\bigotimes_{\beta\in \BB_\gerp} }\pi_1^*\omega_\beta^{e_\gerp p})$.
We prove a lemma.

\begin{lemma}\label{Lemma: q-expansion} Over $Z_{\emptyset,\BB_{\gerp_0}}\subset \tYbarR$, we have $\pi_1^\ast \fbar=({\rm pr}^\ast)^p \pi_2^\ast \gbar$ .
\end{lemma}
\begin{proof} By Lemma \ref{Lemma; cusps}, every connected component of $Z_{\emptyset,\BB_{\gerp_0}}$ contains a cusp of the form $Ta'_\gera$ for some fractional ideal $\gera$ of $\of$. Hence,  it is enough to show the above equality at such cusps.  Let $\gera$ be a fractional ideal of $\of$. Recall from \S \ref{section: q-exp} that $\underline{\omega}$ denotes the canonical generator of the sheaf $\wedge^d\Omega$ on the base of $Ta_\gera$ or $Ta_{\gerp_0\gera}$. Write 
 
 \[
 f(Ta_{\gera})=\sum_{\xi \in (\gera^{-1})^+} a_\xi(\gera) q^\xi \underline{\omega}^p,
\]
\[
 g(Ta_{\gera})=\sum_{\xi \in (\gera^{-1})^+} b_\xi(\gera) q^\xi \underline{\omega}^p.
 \] 
 It follows that 
\[
\pi_1^\ast f(Ta'_\gera)=f(Ta_{\gera})=\sum_{\xi \in (\gera^{-1})^+} a_\xi(\gera) q^\xi \underline{\omega}^p,
\]
\[
{\rm  pr}^\ast \pi_2^\ast g(Ta'_\gera)={\rm pr}^\ast g(Ta_{\gerp_0\gera})=\sum_{\xi \in (\gerp_0\gera^{-1})^+} b_\xi(\gerp_0\gera) q^\xi \underline{\omega}^p.
\]
To prove the result it is enough to show that $b_\xi(p\gera)\in \germ_K$ if $\xi \in
(\gerp_0\gera^{-1})^+-(\gera^{-1})^+$, and  $a_\xi(\gera)-
b_\xi(\gerp_0\gera)\in \germ_K$ for $\xi \in (\gera^{-1})^+$. 

First assume that $\xi \in
(\gerp_0\gera^{-1})^+-(\gera^{-1})^+$, i.e., $\gerp_0 \not |\ \xi\gerp_0\gera$. We can write 
\[
b_\xi(\gerp_0\gera)=\sum_{I\subset \mathbb{S}} (-1)^{|I|} c(\xi \gerp_0\gera,f_I)=\sum_{\gerp_0 \not\in I}(-1)^{|I|}c(\xi \gerp_0\gera,f_I) -\sum_{\gerp_0 \in I} (-1)^{|I|}c(\xi \gerp_0\gera,f_I).
\]
Since for all  primes ideals $\gerl\neq \gerp_0$, the modular forms $f_I$ and $f_{I\cup \gerp_0}$ have $T_\gerl$-eigenvalues that are congruent modulo $\germ_K$, it follows that $c(\xi\gerp_0\gera,f_I)-c(\xi \gerp_0\gera,f_{I \cup \{\gerp_0\}}) \in \germ_K$, whence $b_\xi(\gerp_0\gera) \in \germ_K$.

Now, assume $\xi\in  (\gera^{-1})^+$. We can write
\[
b_\xi(\gerp_0\gera)=\sum_{I \subset \mathbb{S}} (-1)^{|I|}c(\xi
\gerp_0\gera,f_I),
\]
\[
a_\xi(\gera)= \sum_{I \subset \mathbb{S}}
(-1)^{|I|}\tgamma_I c(\xi \gera,f_I)= \sum_{I \subset \mathbb{S}}
(-1)^{|I|}c(\gerp_0,f_I) c(\xi \gera,f_I).
\]

Since for any Hilbert modular eigenform $h$, and  for any integral
ideal $\germ$ of $\of$, we have $c(\gerp_0\germ,h)\equiv c(\gerp_0,h)
c(\germ,h)\ {\rm mod}\ p$, it follows that  $b_\xi(\gerp_0\gera)-a_\xi(\gera) \in \germ_K$.
\end{proof}
 
 We now prove that $\bar{f}$ is divisible by $H$; i.e.,  $\bar{f}/H$ is a mod $p$ Hilbert modular form of parallel weight one. Since irreducible components of distinct $Z_\beta$'s intersect transversally,  it is enough to prove that for each $\gerp|p$, $\fbar$ is divisible by $H_\beta^{e_\gerp}$, for all $\beta \in \BB_\gerp$. We fix $\gerp=\gerp_0$ and show this for all $\beta \in \BB_{\gerp_0}$ working on a stratum on $\tYbarR=\tYbarR(\gerp_0)$.

 Fix $\beta_0 \in \BB_{\gerp_0}$. Let $\pr^{\beta_0,*}=\bigotimes_{\gerp|p}\bigotimes_{\beta\in \BB_\gerp-\{\beta_0\}} (\pr^*_\beta)^{e_\gerp}$. Lemma \ref{Lemma: Isomorphism} proves that the following diagram is commutative over $Z_{\emptyset,\BB_{\gerp_0}}$.
 
 \begin{equation}\label{Diagram: isom2}
 \xymatrix{  \displaystyle{\pi_2^*\omega_{\sigma^{-1}\circ\beta_0}^{p e_{\gerp_0}} \otimes \bigotimes_{\beta \neq \sigma^{-1}\circ\beta_0} \pi_2^*\omega_{\beta}^{p e_{\gerp_0}}}\ar[rr]^{\hspace{7mm}\mu_{\beta_0}^{e_{\gerp_0}} \otimes 1}\ar[rrd]_{(\pr^*)^p} & & \displaystyle{\pi_1^*\omega_{\beta_0}^{e_{\gerp_0}} \otimes \bigotimes_{\beta \neq \sigma^{-1}\circ\beta_0} \pi_2^*\omega_{\beta}^{p e_{\gerp_0}}} \ar[d]^{(\Vbar^*_\beta)^{e_{\gerp_0}} \otimes (\pr^{{\beta_0},*})^p} \\
& & \displaystyle{\pi_1^*\omega_{\sigma^{-1}\circ\beta_0}^{pe_{\gerp_0}} \otimes \bigotimes_{\beta \neq \sigma^{-1}\circ\beta_0} \pi_1^*\omega_{\beta}^{p e_{\gerp_0}}}.} 
 \end{equation}
 Lemma \ref{Lemma: q-expansion} tells us that on $Z_{\emptyset,\BB_{\gerp_0}}$, we have
 
 \begin{eqnarray*}
 \pi_1^*\fbar=(\pr^*)^p\pi_2^*\gbar&=&(\Vbar^*_\beta)^{e_{\gerp_0}} (\mu_{\beta_0}^{e_{\gerp_0}}\otimes (\pr^{{\beta_0},*})^p)(\pi_2^*\gbar)\\
 &=&\pi_1^*H_{\beta_0}^{e_{\gerp_0}}(\mu_{\beta_0}^{e_{\gerp_0}}\otimes (\pr^{{\beta_0},*})^p)(\pi_2^*\gbar)
\end{eqnarray*}
where the last equality follows from Lemma \ref{Lemma: Hasse}. This proves that on the stratum $Z_{\emptyset,\BB_{\gerp_0}}$, the section $\pi_1^*\fbar$ is divisible by $\pi_1^*H_{\beta_0}^{e_{\gerp_0}}$. Since $\pi_1\colon Z_{\emptyset,\BB_{\gerp_0}} \ra \tXbarR$ is a surjective morphism between smooth schemes, it follows that $\fbar$ is divisible by $H_{\beta_0}^{e_{\gerp_0}}$ (If $\fbar/H_{\beta_0}^{e_{\gerp_0}}$ has poles along a divisor, then $\pi_1^*(\fbar/H_{\beta_0}^{e_{\gerp_0}})$ will have poles along the inverse image of that divisor). 

Repeating this argument for all $\gerp|p$, and all $\beta\in \BB_\gerp$, we find that $\fbar$ is divisible by $H=\displaystyle{\prod_{\gerp|p}\prod_{\beta \in \BB_\gerp} H_\beta^{e_\gerp}}$. It follows that $\fbar/H$ is the desired mod $p$ Hilbert modular form of parallel weight $1$.

\end{document}